\documentclass[psamsfonts]{amsart}

\usepackage{amssymb,amsfonts,amscd}

\usepackage[colorlinks,linktocpage]{hyperref}
\hypersetup{linkcolor=[rgb]{0,0,0.715}}
\hypersetup{citecolor=[rgb]{0,0.715,0}}

\newtheorem{thm}{Theorem}[section]
\newtheorem{cor}[thm]{Corollary}
\newtheorem{prop}[thm]{Proposition}
\newtheorem{lem}[thm]{Lemma}

\theoremstyle{definition}
\newtheorem{defn}[thm]{Definition}

\theoremstyle{remark}
\newtheorem{rem}[thm]{Remark}

\makeatletter
\let\c@equation\c@thm
\makeatother
\numberwithin{equation}{section}

\title[]{Some topological results of Ricci limit spaces}

\author[]{Jiayin Pan, Jikang Wang}
\thanks{}

\thanks{J. Pan is partially supported by AMS Simons travel grant.}

\begin{document}
	
\maketitle
\begin{abstract}
	We study the topology of a Ricci limit space $(X,p)$, which is the Gromov-Hausdorff limit of a sequence of complete $n$-manifolds $(M_i, p_i)$ with $\mathrm{Ric}\ge -(n-1)$. Our first result shows that, if $M_i$ has Ricci bounded covering geometry, i.e. the local Riemannian universal cover is non-collapsed, then $X$ is semi-locally simply connected. In the process, we establish a slice theorem for isometric pseudo-group actions on a closed ball in the Ricci limit space. In the second result, we give a description of the universal cover of $X$ if $M_i$ has a uniform diameter bound; this improves a result in \cite{EnnisWei2006}.
\end{abstract}

\section{Introduction}

We study the topology of Ricci limit spaces. Let $(M_{i},p_{i})$ be a sequence of pointed complete Riemannian $n$-manifolds with a uniform Ricci curvature lower bound $\mathrm{Ric}\ge -(n-1)$. Passing to a subsequence if necessary, we can assume $(M_{i},p_{i})$ converges to $(X,p)$ in the pointed Gromov-Hausdorff sense. We call $(X,p)$ a Ricci limit space. $(X,p)$ is non-collapsing if Vol$(B_1(p_i))$ has a uniform lower bound. The regularity theory of Ricci limit spaces has been studied extensively by Cheeger, Colding, and Naber \cite{CheegerColding1997, CheegerColding2000a,CheegerColding2000b,CheegerNaber2013,CheegerNaber2015, ColdingNaber2012}. On the topology of $(X,p)$, Sormani and Wei proved that any Ricci limit space has a universal cover \cite{SormaniWei2001,SormaniWei2004}, but it is unknown whether the universal cover is simply connected by their work. Recently, Wei and the first author proved that any non-collapsing Ricci limit space is semi-locally simply connected \cite{PanWei2019}; in particular, such a limit space has a simply connected universal cover. 
More detailed discussions about the local topology of a Ricci limit space can be found in \cite{PanWei2019}.

In this paper, we prove two topological properties of a Ricci limit space. The first result is that any Ricci limit space with non-collapsing local rewinding volume is semi-locally simply connected.
More precisely, in the collapsing case, if the local universal cover of every $r$-ball in manifolds is non-collapsing, then the limit space is semi-locally simply connected.  

We start with notations in Ricci bounded covering geometry; see \cite{CRX2019,Huang2020,HKRX2020} for further reference. 
Let $B_{r}(p_i)$ be the open ball with center $p_{i}$ and radius $r$. We denote $\bar{B}_r(p_i)$ as the closure of $B_{r}(p_i)$. Let $\widetilde{B_{r}(p_{i})}$ be the (incomplete) universal cover of $B_{r}(p_{i})$ and we choose a point $\tilde{p}_i \in \widetilde{B_{r}(p_{i})}$ such that $\pi (\tilde{p}_i) = p_{i}$. Given the pull-back Riemannian metric on $\widetilde{B_{r}(p_{i})}$, we call the volume of $B_{r/4}(\tilde{p}_i)$ the local rewinding volume of $B_r(p_{i})$, denoted by $\widetilde{Vol}(B_r(p_i))$.

For any metric space $X$ and any $x \in X$, we follow the notation in \cite{PanWei2019} to define the local $1$-contractibility radius:
\begin{equation}\nonumber
	\rho (t,x)= \inf \{ \infty, \rho \geq t | \text{ any loop in } B_{t}(x) \text{ is contractible in } B_{\rho}(x) \}.  
\end{equation}
$X$ is semi-locally simply connected, if for any point $x \in M$, there is an open neighborhood containing $x$ such that any loop in this neighborhood is contractible in $X$. In particular, if for all $x \in M$ there exists $t$ (depending on $x$) such that $\rho(t,x)$ is finite, then $M$ must be semi-locally simply connected.

Now we state our first result:
\begin{thm} \label{semi}
	Let $(M_{i},p_{i})$ be a sequence of complete Riemannian manifolds converging to $(X,p)$ such that for all $i$, \\
	(1) $B_{4}(p_{i}) \cap \partial M_{i}= \emptyset$ and the closure of $B_{4}(p_{i})$ is compact,\\
	(2) $Ric \geq -(n-1)$ on $B_{4}(p_{i})$, $\widetilde{Vol}(B_4(p_{i})) \ge v >0$. \\
	Then $\lim \limits_{t \to 0} \rho(t,x) = 0$ holds for all $x \in B_{1/800}(p)$.
\end{thm}  

Theorem \ref{semi} implies the following.
\begin{cor} \label{t0}
	Let $(M_{i},p_{i})$ be a sequence of complete Riemannian $n$-manifolds with $\mathrm{Ric}\ge -(n-1)$ and converging to $(X,p)$. If there exist $r$ and $v$ such that for all $i$ and any $x \in M_i$, $\widetilde{Vol}(B_r(x)) \ge v > 0$. Then $(X,p)$ is semi-locally simply connected.
\end{cor}

As mentioned, the universal cover of a Ricci limit space always exists \cite{SormaniWei2001,SormaniWei2004}. As a corollary, the universal cover of $X$ in Corollary \ref{t0} is always simply connected. 

To prove Theorem \ref{semi}, we establish the existence of a slice for pseudo-group actions on incomplete Ricci limits. Then Theorem \ref{semi} follows from this slice theorem and Pan-Wei's result \cite{PanWei2019}.

\begin{thm}\label{pseudo slice}
	Let $(M_{i},p_{i})$ be a sequence of connected Riemannian manifolds converging to $(X,p)$ such that for all $i$, $B_{4}(p_{i}) \cap \partial M_{i}= \emptyset$ and the closure of $B_{4}(p_{i})$ is compact. Let $(\widetilde{B_4(p_i)},\tilde{p_i})$ be the universal cover of $B_4(p_i)$ and let $\Gamma_i$ be the fundamental group of $B_4(p_i)$. The pseudo-group
	$$G_i=\{\gamma\in \Gamma_i| d(\gamma\tilde{p}_i,\tilde{p_1})\le 1/100 \}$$
	acts on $B_1(\tilde{p}_i)$. Passing to a subsequence, $(\bar{B}_1(\tilde{p}_i), \tilde{p}_i, G_i)$ converges to $(\bar{B}_1(\tilde{p}),\tilde{p},G)$ in the equivariant Gromov-Hausdorff sense. Then the following holds.
	
	For any $\tilde{x} \in B_{1/800}(\tilde{p})$, there is a slice $S$ at $\tilde{x}$ for pseudo-group $G$-action, that is, $S$ satisfies:\\
	(1) $\tilde{x}\in S$ and $S$ is $G_{\tilde{x}}$-invariant, where $G_{\tilde{x}}$ is the isotropy subgroup at $\tilde{x}$;\\
	(2) $S/G_{\tilde{x}}$ is homeomorphic to a neighborhood of $x$. 
\end{thm}

Note that we don't assume $B_{1}(\tilde{p}_i)$ has a volume lower bound in Theorem \ref{pseudo slice}. 

The idea to find the slice is that we \textit{extend} $G$ to a Lie group $\hat{G}$, which acts homeomorphically on an \textit{extended} space. Then we apply Palais's slice theorem \cite{Palais1961} for this $\hat{G}$-action. In this construction, we will show that the extended group and space are locally homeomorphic old ones, thus we can get the slice at $\tilde{x}$. See Sections 4 and 5 for details.

The next result of this paper is about the description of universal covers of Ricci limit spaces. Now we assume that $\mathrm{diam}(M_i)\le D$ for some $D>0$. Let $\widetilde{M_i}$ be the universal cover of $M_i$ with fundamental group $\Gamma_i$. Passing to a subsequence, we obtain the following convergence
\begin{center}
	$\begin{CD}
		(\widetilde{M}_i,\tilde{p}_i,\Gamma_i) @>GH>> (Y,\tilde{p},G)\\
		@VV\pi V @VV\pi V\\
		(M_i,p_i) @>GH>> (X,p),
	\end{CD}$
\end{center}
Ennis and Wei showed that there is $\epsilon=\epsilon(X)>0$ such that
$$(\widetilde{M}_i,\tilde{p}_i,\Gamma_i^\epsilon)\overset{GH}\longrightarrow(Y,\tilde{p},H^\epsilon)$$
and the universal cover of $X$ is isometric $Y/H^\epsilon$ \cite{EnnisWei2006}, where $\Gamma_i^\epsilon$ is generated by elements
$$\{ g\in \Gamma_i\ |\ d(q,gq)\le \epsilon \text{ for some } q\in \widetilde{M}_i\}.$$
We give a description of $H^\epsilon$ from the $G$-action on $Y$.

\begin{thm} \label{desc_cover}
	Let $M_i$ be a sequence of closed Riemannian $n$-manifolds with
	$$\mathrm{Ric}(M_i)\ge -(n-1),\quad \mathrm{diam}(M_i)\le D.$$
	Suppose that
	\begin{center}
		$\begin{CD}
			(\widetilde{M}_i,\tilde{p}_i,\Gamma_i) @>GH>> (Y,\tilde{p},G)\\
			@VV\pi_i V @VV\pi V\\
			(M_i,p_i) @>GH>> (X,p).
		\end{CD}$
	\end{center}
	Let $H$ be the group generated by $G_0$ and all isotropy subgroups of $G$, where $G_0$ is the identity component subgroup of $G$. Then $Y/H$ is the universal cover of $X$. Consequently, $\pi_1(X,p)$ is isomorphic to $G/H$. 
\end{thm}

Note that $H$ is a subgroup of $H^\epsilon$. Theorem \ref{desc_cover} shows that they are indeed the same. This improvement relies on the fact that $\mathrm{Isom}(Y)$ is a Lie group \cite{CheegerColding2000a,ColdingNaber2012}.

We organize the paper as follows. In Section 2, we recall some preliminaries, which include Ricci limit spaces and Palais's slice theorem. In Section 3, we study the limit pseudo-group actions on a closed ball of the Ricci limit space. To prove Theorem \ref{pseudo slice} next, we will use the limit pseudo-group $G$ to construct a Lie group $\hat{G}$ in Section 4, then we construct a $\hat{G}$-space and apply Palais's theorem to find a slice on the new space in Section 5. We prove Theorem \ref{desc_cover} in Section 6.

The authors would like to thank Xiaochun Rong for many helpful discussions.

\tableofcontents

\section{Preliminaries}

\subsection{Ricci limit spaces}

One important result about Ricci limit spaces that we need in this paper is their isometry groups are always Lie groups. \cite{CheegerColding2000a} proves the non-collapsing case, then \cite{ColdingNaber2012} proves the general case.
\begin{thm}\cite{CheegerColding2000a,ColdingNaber2012}\label{isom_lie}
	The isometry group of any Ricci limit space is a Lie group.
\end{thm}

We recall some of elements in the proof of Theorem \ref{isom_lie} since we need them later in Section 3.  

Let $(X,p)$ be a Ricci limit space. For any integer $k \le n$, $\mathcal{R}_k$ is the set of points where every tangent cone is isometric to $\mathbb{R}^k$. We say $y \in (\mathcal{R}_{k})_{\epsilon, \delta}$, if for all $0 < r < \delta$ we have $d_{GH}(B_r(y), B_r^{k}(0)) < \epsilon r$, where $B_r^{k}(0)$ is the $r$-ball in $\mathbb{R}^k$. Let $(\mathcal{R}_{k})_{\epsilon} = \cup_{\delta} (\mathcal{R}_{k})_{\epsilon, \delta}$. Note that $\mathcal{R}_k=\cap_{\epsilon>0} (\mathcal{R}_k)_\epsilon$.

A topological group is said to have small subgroups if every neighborhood of the identity contains a non-trivial subgroup. By \cite{Gleason1952, Yamabe1953}, a topological group with no small groups is a Lie group. For the isometry group of a metric space $X$, we can measure the smallness of a subgroup by its displacement. For any subgroup $H$ of $\mathrm{Isom}(X)$ and $x \in X$, we define 
$$\rho_H(x)=\sup_{h \in H} d(x, hx),\quad D_{H,r}(x)= \sup_{w \in B_{r}(x)} \rho_H(w)$$ be the displacement function of $H$ on $B_r(x)$. A standard fact is that there is no non-trivial group $H \subset$ Isom$(\mathbb{R}^{k})$ such that $D_{H,1}(0) \le 1/20$, where $0 \in \mathbb{R}^{k}$. 

Using certain path connectedness property of $(\mathcal{R}_{k})_{\epsilon, \delta}$, where $k$ is the dimension of $X$ in the sense of \cite{ColdingNaber2012}, the proof of Theorem \ref{isom_lie} rules out a sequence of nontrivial subgroups $H_i$ with $D_{H_i,R}(z)\to 0$ for all $R>0$ and all $z\in X$. A similar idea is applied in \cite[Lemma 2.1]{PanRong2018} to prove the following.

\begin{prop}\cite{PanRong2018}\label{PR_nofoxball}
	Let $X$ be a Ricci limit space and let $g$ be an isometry of $X$. If $g$ fixes every point in $B_1(x)$, where $x\in X$, then $g$ fixes every point in $X$.
\end{prop}

We will prove a local version of Proposition \ref{PR_nofoxball} in Section 3, which is essential to study pseudo group actions. The main difference is that we are considering limit space that comes from a sequence of incomplete manifolds.

We also recall the main theorem in \cite{PanWei2019}:
\begin{thm}\label{PW}
	Let $(M_{i},p_{i})$ be a sequence of connected Riemannian $n$-manifolds (not necessarily complete) converging to $(Y,p)$ such that for all $i$, \\
	(1) $B_{4}(p_{i}) \cap \partial M_{i}= \emptyset$ and the closure of $B_{4}(p_{i})$ is compact,\\
	(2) $\mathrm{Ric} \geq -(n-1)$ on $B_{4}(p_{i})$, $Vol(B_{4}(p_{i})) \geq v >0$. \\
	Then
	$\lim \limits_{t \to 0} \rho(t,x)=0$
	holds for any $x \in B_{2}(p)$. 
\end{thm}
In fact, a stronger result $\lim\limits_{t\to 0} \rho(t,x)/t=1$ is proved in \cite{PanWei2019}.

\subsection{Palais's slice theorem}

We review the definition of a slice and Palais's theorem \cite{Palais1961} in this subsection.

Let $X$ be a completely regular topological space and let $G$ be a Lie group. We say that $X$ is a $G$-space, if $G$ acts on $X$ by homeomorphisms. For any $x \in X$, we define isotropy subgroup $G_{x}= \{ g \in G | \ gx=x \}$. If $G$-action on $X$ is proper, then $G_x$ is compact for all $x\in X$. We say a subset $S\subset X$ is $G_x$-invariant, if $G_x S=S$. For $G_x$-invariant set $S$, we can define an equivalence relation on $G \times S$ by claiming $(g,s) \sim (gh^{-1}, hs)$ for all $g \in G, h \in G_{x}, s \in S$.
We denote the corresponding quotient space as 
$$G \times_{G_x} S = G \times S / \sim.$$ 
There is a natural left $G$ action on $G \times_{G_x} S$ given by $g[g',s]=[gg',s]$, where $g \in G$ and $[g',s] \in G \times_{G_x} S$. 

\begin{defn}
   We call $S \subset X$ a $G_x$-slice (briefly, a slice) at $x$ if the followings hold: \\
(1) $x\in S$ and $S$ is $G_x$-invariant; \\
(2) $GS$ is an open neighborhood of $x$; $\psi: [g,s] \mapsto gs$ is a $G$-homeomorphism between $G \times_{G_x} S$ and $GS$, that is, $\psi$ is a homeomorphism and $\psi \circ g = g \circ \psi$ for all $g \in G$.
\end{defn}

In particular, if $S$ is a slice at $x$, then $S/G_x$ is homeomorphic to $GS/G$.


The following slice theorem is from \cite{Palais1961}.
\begin{thm}\label{topo_slice}
	Let $X$ be a $G$-space and $x \in X$, where $G$ is a Lie group. Then the following two conditions are equivalent: \\
	(1) $G_{x}$ is compact and there is a slice at ${x}$. \\
	(2) There is a neighborhood $U$ of $x \in X$ such that $\{g \in G | gU \cap U \neq \emptyset \}$ has compact closure in $G$. 
\end{thm}

It is well-known that for an isometric $G$-action on a locally compact metric space $X$, (or more generally, proper actions,) the second statement in theorem \ref{topo_slice} is always fulfilled, thus a slice always exists. We include the proof here for readers' convenience. 

\begin{cor}\label{slice}
	Let $X$ be a locally compact metric space. Let $G$ be a Lie group acting isometrically and effectively on $X$. Then for any $x \in X$, there is a $G_x$-slice at $x$.
\end{cor}
\begin{proof}
	We claim that $G(1)=\{ g \in G | \ d(gx,x) \le 1 \}$ is compact. If this claim holds, then we set $U=B_{1/2}(x)$. With this $U$, $\{ g\in G | gU \cap U\not= \emptyset \}$ is contained in $G(1)$, thus it has a compact closure in $G$. Then the result follows from Theorem \ref{topo_slice}.
	
	We prove the claim. Let $\{g_j\}$ be any sequence in $G(1)$. We shall find a converging subsequence. Let $\{ x_i \}$ be a sequence that is dense in $X$.  For $x_1$, we have 
	$$	d(x,g_j x_1) \le d(x, g_j x) + d(g_j x, g_j x_1) \le 1 + d(x, x_1).$$   
	So we may assume $g_j x_1$ converges to some $y_1$ by passing to a subsequence. For this subsequence, we repeat this process for $x_2$,...,$x_i$, and by a standard diagonal argument, we may assume that for all $i$, $g_j x_i$ converges to $y_i$ as $j \to \infty$. We define $g:X\to X$ by setting $g(x_i)=y_i$ and extending $g$ to the whole space continuously. By construction, $g$ is an isometry of $X$ and  $g_j$ converges to $g$. Moreover, $g\in G(1)$ because $d(x,gx)=\lim d(x,g_jx)\le 1.$
\end{proof}

\subsection{Loop Liftings}
In this subsection, we explain how we apply Theorem \ref{PW} and the slice theorem to prove semi-local simple connectedness. We consider an easier case that the diameter of $M_i$ is strictly less than $4$. In this case, the local universal cover $\widetilde{B_4(p_i)}$ is the same as universal cover $\widetilde{M}_i$ of $M_i$.  For consistency, We still use the notation $B_4(p_i)$ instead of $M_i$. Let $G_i$ be the fundamental group of $B_4(p_i)$. Passing to a subsequence if necessary, we have
\begin{equation} \label{e1}
	\begin{CD}
		(\widetilde{B_4(p_i)},\tilde{p}_i,G_i) @>GH>> (\widetilde{X},\tilde{p},G)\\
		@VV\pi V @VV\pi V\\
		(B_4(p_i),p_i) @>GH>> (X= \widetilde{X}/G,p),
	\end{CD}
\end{equation}
Since $x \in B_{1/800}(p)$, there exists $x_{i} \in M_i$ converging to $x$. We change the base points and consider
\begin{equation} \label{e2}
	\begin{CD}
		(\widetilde{B_4(p_i)},\tilde{x}_i,G_i) @>GH>> (\widetilde{X},\tilde{x},G)\\
		@VV\pi V @VV\pi V\\
		(B_4(p_i),x_i) @>GH>> (X= \widetilde{X}/G,x).
	\end{CD}
\end{equation}

To prove $\lim \limits_{t \to 0} \rho(t,x) = 0$, we lift a loop based at $x$ to a loop based at $\tilde{x}$; see diagram \ref{e2}. This idea is also used in \cite{PanWei2019}, where the case $G\tilde{x}=\tilde{x}$ is considered. Note that the loop is only continuous and may not be rectifiable. For a general non-compact Lie group $G$-action, we don't know whether a path-lifting exists. Even though such path-lifting exists, a loop may be lifted to a non-closed path since in general $G\tilde{x}\not= \tilde{x}$. We make use of the slice theorem to lift any loop based at $x$ to a loop based at $\tilde{x}$. Furthermore, a loop in a small ball centered at $x$  must be lifted in a small ball centered at $\tilde{x}$ since $\pi : S/G_{\tilde{x}} \to GS/G$ is a homeomorphism.

Now we follow this idea to prove Theorem \ref{semi} when $\mathrm{diam}(M_i)<4$ for all $i$:
\begin{proof}
	Consider $(\tilde{X},\tilde{x},G)$ as in the diagram \ref{e2}. Given any $\epsilon>0$, since $\lim_{t \to 0} \rho (t,\tilde{x}) = 0$, there exists $\delta'>0$ such that $\rho (t, \tilde{x}) < \epsilon$ holds for all $t \leq \delta'$. $G$ is a Lie group due to \cite{CheegerColding2000a}. By Corollary \ref{slice}, there is a slice $S$ at $\tilde{x}$.
	
	Since $GS$ is open, $GS/G$ contains a open ball $B_{\delta}(x)$. We show that 
	$$\pi^{-1} (B_{\delta}(x)) \cap S \subset B_{\delta'}(\tilde{x})$$ 
	when $\delta$ is small enough. Otherwise, there exist $\delta_j \to 0$ and $\tilde{x}_j \in \pi^{-1} (B_{\delta_j}(x)) \cap S$ but $\tilde{x}_j \notin B_{\delta'} (\tilde{x})$. While $\pi (\tilde{x}_{j}) \in B_{\delta_j}(x)$ converges to $x=\pi (\tilde{x})$, $\tilde{x}_j$ does not converge to $\tilde{x}$. This contradicts to that $\pi : S/G_{\tilde{x}} \to GS/G$ is a homeomorphism. 
	
	Now we have $\delta>0$ small such that $\pi^{-1} (B_{\delta}(x)) \cap S \subset B_{\delta'}(\tilde{x})$. Choose any loop $\gamma$ in $B_{\delta}(x)$. Since slice theorem applies, we can lift $\gamma$ to loop $\tilde{\gamma}$ based at $\tilde{x}$ and contained in $S$. Actually the image of $\tilde{\gamma}$ is also contained in $B_{\delta'}(\tilde{x})$ since $\pi^{-1} (B_{\delta}(x)) \cap S \subset B_{\delta'}(\tilde{x})$.  
	
	Because $\rho (\tilde{x}, \delta') < \epsilon$, there exists a homotopy $H$, of which the image is contained in $B_{\epsilon}(\tilde{x})$, from $\tilde{\gamma}$ to the constant loop at $\tilde{x}$. Then $\pi \circ H$  is a homotopy from  $\gamma$ to the constant loop at $x$ and the image of this homotopy is in $B_{\epsilon}(x)$. So any loop in $B_{\delta}(x)$ is contractible in $B_{\epsilon}(x)$, i.e. $\rho(t,x) \leq \epsilon$ when $t<\delta$. Since $\epsilon$ can be arbitrarily small, $\lim_{t \to 0} \rho(t,x) = 0$.
\end{proof} 

\section{Pseudo-group}

We no longer assume that $M_i$ has diameter less than $4$. $\widetilde{B_4(p_i)}$ may not have a Gromov-Hausdorff limit in general (see \cite{SormaniWei2004}); instead, we will use $\bar{B}_1(\tilde{p}_i)$ in $\widetilde{B_4(p_i)}$. Choose $\tilde{p_i} \in \widetilde{B_4(p_i)}$ with $\pi(\tilde{p}_i)=p_i$. We consider the closed ball $\bar{B}_{1}(\tilde{p}_i)$ with intrinsic metric in $\widetilde{B_4(p_i)}$ and $$G_{i}= \{ g \in \Gamma_i | d(gp_i,p_i) \le 1/100 \}.$$ Passing to a subsequence, 
$(\bar{B}_{1}(\tilde{p}_i),\tilde{p}_i)$ converges to $(\bar{B}_{1}(\tilde{p}),\tilde{p})$ by Bishop-Gromov relative volume comparison. We will define the meaning of convergence $G_i \to G$, where the limit $G$ is a pseudo-group, which means $gg' \notin G$ for some $g,g' \in G$. By the argument of last subsection, to prove Theorem \ref{semi}, it is enough to prove the existence of a slice for pseudo-group actions.

Now we start to construct the pseudo-group $G$. Let $\Gamma_i$ be the fundamental group of $B_4(p_i)$. 
Define 
\begin{equation*}
	G_{i}'=\{ g \in \Gamma_i | d(g\tilde{p}_i,\tilde{p}_i) \le 1/30 \},\ \ \ G_{i}= \{ g \in \Gamma_i | d(g\tilde{p}_i,\tilde{p}_i) \le 1/100 \} \subset G_i'.
\end{equation*} 

Notice that elements in $G_i'$ or $G_i$ preserve the distance on $\widetilde{B_4(p_i)}$ but may not preserve intrinsic metric on $\bar{B}_{1}(\tilde{p}_i)$. However, the intrinsic metric $d$ and restricted metric (from $\widetilde{B_4(p_i)}$) on $\bar{B}_{1}(\tilde{p}_i)$ coincide on $\bar{B}_{1/3}(\tilde{p}_i)$. So for any $z_1,z_2 \in \bar{B}_{1/10}(\tilde{p}_i)$ and $g \in G_i'$, we have $gz_1,gz_2 \in  \bar{B}_{2/15}(\tilde{p}_i) \subset \bar{B}_{1/3}(\tilde{p}_i)$ and thus $d(z_1,z_2)=d(gz_1,gz_2)$. Abusing the notation, we use notation $d$ for intrinsic metrics on $\bar{B}_{1}(\tilde{p}_i)$ and $\bar{B}_{1}(\tilde{p})$. For $r \le 1$, we restrict metric $d$ from $\bar{B}_{1}(\tilde{p})$ (or $\bar{B}_{1}(\tilde{p}_i)$) to $\bar{B}_{r}(\tilde{p})$ (or $\bar{B}_{r}(\tilde{p}_i)$).  

We see $G_i'$ and $G_i$ as sets of distance-preserving maps from $\bar{B}_{1/10}(\tilde{p}_i)$ to $\bar{B}_{1/3}(\tilde{p}_i)$. After passing to a subsequence, $G_i'$ converges to $G'$ and $G_i$ converges to $G$, where $G'$ and $G$ are closed subsets of distance-preserving maps from $\bar{B}_{1/10}(\tilde{p})$ to $\bar{B}_{1/3}(\tilde{p})$. Though Theorem \ref{pseudo slice} is stated with $\bar{B}_{1/800}(\tilde{p})$ and $G$, we need $\bar{B}_{1/10}(\tilde{p})$ and $G'$ to provide extra room to study $\bar{B}_{1/800}(\tilde{p})$ and $G$.

We consider $G'$ at first. For all $g \in G'$, $d(g\tilde{p}, \tilde{p}) \le 1/30$ holds. Notice that any $g \in G'$ is only assumed to be defined on $\bar{B}_{1/10}(\tilde{p})$. 
$G'$ has compact-open topology as a set of continuous maps $C(\bar{B}_{1/10}(\tilde{p}),\bar{B}_{1/3}(\tilde{p}))$. $G'$ is compact by the same proof of Corollary \ref{slice}. However, there is no group structure so far since $g_1g_2$ may not be defined on $\bar{B}_{1/10}(\tilde{p})$ for $g_1, g_2 \in G'$.

Before endowing $G'$ a pseudo-group structure, we prove a lemma as below for later use.

\begin{lem} \label{lemma_nofix}
	For any $g \in G'$, if $g$ fix $B_r(\tilde{p})$ for some $0 < r < 1/10$, $g$ must be the identity map on $\bar{B}_{1/10}(\tilde{p})$.
\end{lem}

As mentioned in Section 2, this lemma is a local version of Lemma 2.1 in \cite{PanRong2018}. The main difference is that we are considering the limit space coming from a sequence of incomplete manifolds, so our construction should keep away from the boundary. Also see \cite{CheegerColding2000a} and \cite{ColdingNaber2012}.

For any $\tilde{x} \in B_{1/10}(\tilde{p})$ and $r$ such that $B_r(\tilde{x}) \subset B_{1/10}(\tilde{p})$, we write $$D_{H,r}(x)= \sup_{w \in B_{r}(x)} \rho_H(w)$$ as the displacement function of some subgroup $H$ on $B_r(x)$.

\begin{proof}[Proof of Lemma \ref{lemma_nofix}]
	Let $k$ be the dimension of $B_1(\tilde{p})$ in the sense of \cite{ColdingNaber2012}. We assume $r$ is the largest $t$ such that $g$ fixes $B_t(\tilde{p})$.
	Let $C$ be the closure of subset generated by $g$ in $G'$. $C$ fixes $B_r(\tilde{p})$ as well and thus $C$ is a group. Let $s= (1/10 - r)/4$. We shall show for all $\epsilon>0$, there exist $x(\epsilon)\in (\mathcal{R}_{k})_{\epsilon, \tau (\epsilon)} \cap B_{r + 3s}$ and $s > \tau (\epsilon)>r(\epsilon)>0$, such that 
	\begin{equation}\label{displacement_1/20}
		D_{C, r(\epsilon)}(x(\epsilon))= r(\epsilon)/20
	\end{equation}
	If so, then we can find a sequence $\epsilon_l \to 0$, where $l\in\mathbb{N}$, such that 
	\begin{equation}
		(B_{r(\epsilon_l)}(x(\epsilon_l)), (r(\epsilon_l))^{-1}d, C) \to (B^k_1(0), d_0, H).
	\end{equation} 
	$H$ is non-trivial with $D_{H,1}(0) = 1/20$, which is impossible.
	
	Fixing $\epsilon$, we verify (\ref{displacement_1/20}). Since $r$ is the largest $t$ such that $g$ fix $B_t(\tilde{p})$, we can find $w \in  B_{s+r} (\tilde{p}) \cap \mathcal{R}_k$ such that $g(w) \neq w$. We claim any geodesic from $w$ to a point in $B_{s}(\tilde{p})$ is contained in $B_{1/10-s}(\tilde{p})$. Otherwise there exists a geodesic from $w$ to $x \in B_{s}(\tilde{p})$ and there is point $y$ on this geodesic with $d(\tilde{p},y)=1/10 - s$. Then $d(x,w)=d(x,y)+d(y,w) > 1/10 -2s + 2s = 1/10$ and $d(x,w) \le d(w, \tilde{p}) + d(\tilde{p},x) \le 1/10-3s+s=1/10-2s$, which is a contradiction.
	
	Choose $\theta < \min \{ d(w, gw), s \}$ such that $w \in (\mathcal{R}_k)_{\epsilon, 2\theta}$. Then we have $D_{C, \theta}(w) \ge \theta/20$. Let $z \in (\mathcal{R}_k)_{\epsilon, 2\eta}$ be a regular point in $B_{s}(\tilde{p}) \cap B_{r}(\tilde{p})$ for small $\eta$. $D_{C, \eta}(z)= 0 \le \eta/20$ since $C$ fix $B_{r}(\tilde{p})$. Let $\lambda <\min( \theta, \eta)$. 
	
	Notice that $D_{C,r}(z)$ is continuous in $r$ and $z$ when $B_r(z)$ is away from the boundary of $\bar{B}_{1/10}(p)$. Suppose that $D_{C,\lambda}(z) > \lambda/20$, by intermediate value theorem, we can find $\lambda \le r \le \eta$ such that $D_{C,r}(z) = r/20$.
	
	Similarly, if $D_{C, \lambda}(w) < \lambda/20$, we can find $\lambda \le r \le \theta$,$D_{C,r}(w) = r/20$.
	
	Now we assume $D_{C, \lambda}(z) \le \lambda/20 \le D_{C,\lambda}(w)$. By \cite{CheegerColding2000a} for $k=n$ and \cite{ColdingNaber2012} for general $k$, we may assume there exists $\delta$ and a geodesic $c: [0,l] \to B_{1/10-s}(\tilde{p})$ from $w$ to $z$ such that $c$ is contained in $(\mathcal{R}_k)_{\epsilon, \delta}$. Notice that we just showed that a geodesic from $w$ to $z$ must be contained in $B_{1/10-s}(\tilde{p})$, so $c \subset B_{1/10-s}(\tilde{p})$. By intermediate value theorem we can find $z(\epsilon) \in (\mathcal{R}_k)_{\epsilon , \delta}$ on $c$ such that $D_{C,\lambda}(z(\epsilon)) = \lambda/20$. 
\end{proof}

\begin{cor}\label{Cor1}
	For any $\delta>0$ and $r<1/10$, there exists $\epsilon$ so that the following holds: if $g_i,g_i' \in G_i'$ such that $d(g_i \tilde{x},g_i' \tilde{x})<\epsilon$ for all $\tilde{x} \in \bar{B}_{r}(\tilde{p}_i)$, we have $d(g_iw,g_i'w)<\delta$ for all $w \in \bar{B}_{1/10}(\tilde{p}_i)$. 
\end{cor}
\begin{proof}
	Assume there exist $\delta>0$, $r<1/10$, $g_i,g_i' \in G_i'$ such that $d(g_i \tilde{x},g_i' \tilde{x})<\epsilon_i \to 0$ for all $\tilde{x} \in \bar{B}_{r}(\tilde{p}_i)$, but $d(g_iw_i,g_i'w_i) \ge \delta$ for some $w_i \in \bar{B}_{1/10}(\tilde{p}_i)$. For large $i$, $g_i'g_i^{-1} \in G_i'$  since $d(g_i'g_i^{-1} \tilde{p}, \tilde{p}) < \epsilon_i$. By passing to a subsequence, $g_i'g_i^{-1} \to g$ and $w_i \to w \in \bar{B}_{1/10}(\tilde{p})$. $g$ fix $\bar{B}_{r}(\tilde{p})$ but $d(gw,w) \ge \delta$. This contradicts to the lemma above.  
\end{proof}

Now we give $G'$ a pseudo-group structure as follows. For any $g_0,g_1,g_2 \in G'$, notice that $g_2(\bar{B}_{1/250}(\tilde{p})) \subset \bar{B}_{1/10}(\tilde{p})$; so $g_1g_2(\bar{B}_{1/250}(\tilde{p}))$ is always well-defined. We say $g_0=g_1g_2$ if $g_0(z)=g_1g_2(z)$ for all $z \in \bar{B}_{1/250}(\tilde{p})$. By Lemma \ref{lemma_nofix}, such $g_0$ is unique if it exists. This gives a pseudo-group structure on $G'$. The identity map defined on $\bar{B}_{1/10}(\tilde{p})$, which we write as $1$, is the identity element. 
\begin{rem}
	We remind readers that we view elements in $G'$ as continuous maps from $\bar{B}_{1/10}(\tilde{p})$ to $\bar{B}_{1/3}(\tilde{p})$. We check the product of the pseudo-group structure on $G'$ by checking their values on $\bar{B}_{1/250}(\tilde{p})$.
\end{rem}

We claim that for any $g \in G'$, it has an inverse $g^{-1} \in G'$. We can choose $g_i \in G_i'$ such that $g_i$ converges to $g$. Since $d(g_i^{-1}\tilde{p}_i,\tilde{p}_i)=d(\tilde{p}_i,g_i\tilde{p}_i) \le 1/30$, $g^{-1}_i \in G_i'$. After passing to a subsequence, we assume $g^{-1}_i \to g' \in G'$. For any $z \in \bar{B}_{1/250}(\tilde{p})$, we can find $z_i$ converging to $z$ where $z_i \in B_{1/240}(\tilde{p}_i)$. Then $gg'(z)$ is well-defined since $g'(z) \in \bar{B}_{1/10}(\tilde{p})$. $z_i = g_ig_i^{-1} z_i \to gg'z$, so $gg'(z)=z$ for all $z \in \bar{B}_{1/250}(\tilde{p})$. Thus $g'=g^{-1}$ in $G'$. 

A similar proof shows that for any $g_1,g_2 \in G'$ such that $d(g_1g_2(\tilde{p}),\tilde{p})<1/30$, we have $g_1g_2 \in G'$. Notice that $g_1g_2(\tilde{p})$ means $g_1 (g_2(\tilde{p}))$, which is always defined even though $g_1g_2$ is not in $G'$. 

For $G$, we can define a pseudo-group structure as above. Recall that any $g\in G$ satisfies $d(g\tilde{p}, \tilde{p}) \le 1/100$, thus $G$ is a subset of $G'$. It is clear that the followings hold:\\
(1) if $g \in G$, we have $g^{-1} \in G$; \\
(2) $g_1g_2g_3 \in G'$ for all $g_1,g_2,g_3 \in G$, since $d(g_1g_2g_3 \tilde{p},\tilde{p})<1/30$. 

We give a simple but useful lemma below:
\begin{lem}\label{lemma4}
	Let $g_1,g_2 \in G$. If there exists $x \in \bar{B}_{1/240}(\tilde{p})$ such that $g_1g_2(x) \in \bar{B}_{1/240}(\tilde{p})$, then $g_1g_2 \in G$.
\end{lem}
\begin{proof}
	Since $g_1,g_2 \in G$, $g_1g_2 \in G'$. Because $x \in \bar{B}_{1/240}(\tilde{p})$ and $g_1(g_2(x)) \in \bar{B}_{1/240}(\tilde{p})$, $d(g_1g_2\tilde{p},\tilde{p})<1/100$, we have $g_1g_2 \in G$ as well.
\end{proof}

We define $\bar{B}_{1/250}(\tilde{p})/G$ as the quotient set, where the equivalent relation is given by $z \sim gz$ for any $z \in \bar{B}_{1/250}(\tilde{p})$ and $g \in G$ such that $gz \in \bar{B}_{1/250}(\tilde{p})$.  We define a natural quotient metric on $\bar{B}_{1/250}(\tilde{p})/G$. For any $[z],[w] \in \bar{B}_{1/250}(\tilde{p})/G$, define $d([z],[w])$ be the minimum of $d(z,w)$ where $z,w$ are pre-images of $[z],[w]$. By \cite{FukayaYamaguchi1992}, $\bar{B}_{1/250}(\tilde{p}_i)/G_i$, with the quotient metric as above, converges to $\bar{B}_{1/250}(\tilde{p})/G$ .

It is plausible that $\bar{B}_{1/250}(\tilde{p})/G$ is isometric to $\bar{B}_{1/250}(p)$, but the following observation shows we can only get an isometry between $B_{1/800}([\tilde{p}])$ and $B_{1/800}(p)$ where $[\tilde{p}] = \pi (\tilde{p}) \in \bar{B}_{1/250}(\tilde{p})/G$. 
\begin{rem}\label{observation}
	For any $g \in \Gamma_i$ with $g \bar{B}_{1/250}(\tilde{p}_i) \cap \bar{B}_{1/250}(\tilde{p}_i)$ is non-empty, we have $g \in G_i$. As a result, $\bar{B}_{1/250}(\tilde{p}_i)/G_i$ is homeomorphic to $\bar{B}_{1/250}(p_i)$. However, the quotient metric on $\bar{B}_{1/250}(\tilde{p}_i)/G_i$ may not be equal to the metric on $\bar{B}_{1/250}(p_i)$ for the following reason. For any $x_1,x_2 \in \bar{B}_{1/250}(p_i)$, $d(x_1,x_2)$ equals $\min d(\tilde{x}_1,\tilde{x}_2)$ where $\tilde{x}_1,\tilde{x}_2 \in \widetilde{B_{4}(p_i)}$ and $\pi(\tilde{x}_1)=x_1, \pi(\tilde{x}_2)=x_2$. It's possible for any pair of $\tilde{x}_1,\tilde{x}_2$ on which we get the minimum, one of them is out of $\bar{B}_{1/250}(\tilde{p}_i)$, then the quotient metric would be larger. But two metric coincide on $x_1,x_2 \in \bar{B}_{1/800}(p_i)$, since we can choose $\tilde{x}_1 \in \bar{B}_{1/800}(\tilde{p}_i)$ and there must be a $\tilde{x}_2 \in \bar{B}_{1/400}(\tilde{x}_1) \subset \bar{B}_{1/250}(\tilde{p}_i)$ such that $d(x_1,x_2)=d(\tilde{x}_1,\tilde{x}_2)$. As a result, $1/800$-ball at $[\tilde{p_i}]=\pi(\tilde{p_i}) \in \bar{B}_{1/250}(\tilde{p}_i) / G_i$ is isometric to  $\bar{B}_{1/800}(\tilde{p}_i)$. Passing to the limit, $B_{1/800}([\tilde{p}])$ is isometric to $B_{1/800}(p)$.
\end{rem}

\section{Groupfication of $G$}
We need a Lie group action to apply the slice theorem. However, $G$ is only a pseudo-group. To overcome this problem, we will first construct a Lie group $\hat{G}$ out of $G$. This process is related to \cite{FukayaYamaguchi1992}.

Let $F_{G}$ be the free group generated by $e_g$ for all $g \in G$. We quotient $F_G$ by the normal subgroup generated by all elements of the form $e_{g_1}e_{g_2}e_{g_1g_2}^{-1}$, where $g_1,g_2 \in G$ with $g_1g_2 \in G$. The quotient group is denoted as $\hat{G}$. We write elements in $\hat{G}$ as $\hat{g}= [e_{g_1}e_{g_2}...e_{g_k}]$.

\begin{lem}\label{lem_inj}
	The natural map $i:G \to \hat{G}$, $g\mapsto [e_g]$ is injective.
\end{lem}
\begin{proof}
	Assuming there exists $g \in G$ such that $[e_g]=1$, we shall show that $g$ is the identity map on $B_{1/10}(\tilde{p})$. By the definition, there exist $g_j^l, g_{j1}, g_{j2} \in G$ with $g_{j1}g_{j2} \in G$ such that the following holds in the free group $F_G$:
	\begin{equation} \label{e3}
		e_g= \prod_{j=1}^{k} ((\prod_{l=1}^{k_j} e_{g_j^l}^{s_j^l}) (e_{g_{j1}}e_{g_{j2}}e_{g_{j1}g_{j2}}^{-1})^{s_j} (\prod_{l=k_j}^{1} e_{g_j^l}^{-s_j^l})),
	\end{equation}
	where each $s_j$ and $s_j^l$ is  $1$ or $-1$. We see the right-hand side as a sequence of words. 
	The equality means we can find a way to keep canceling adjacent terms and get $e_g$ eventually.
	
	Let $\phi_i: G \to G_i \hookrightarrow \Gamma_i$ be the equivariant GH approximation. Although $\phi_i$ generally is not a homomorphism onto the image, we claim that 
	\begin{equation} \label{e4}
		\phi_i(g)= \prod_{j=1}^{k} ((\prod_{l=1}^{k_j} \phi_i(g_j^l)^{s_j^l}) (\phi_i(g_{j1})\phi_i(g_{j2})\phi_i(g_{j1}g_{j2})^{-1})^{s_j} (\prod_{l=k_j}^{1} \phi_i(g_j^l)^{-s_j^l})).
	\end{equation}
	Notice the right-hand side of (\ref{e4}) comes from (\ref{e3}) by the following procedure: replace $e_h$ by $\phi_i(h)$ and replace $e_h^{-1}$ by $\phi_i(h)^{-1}$ which is the inverse of $\phi_i(h)$. In particular, if two terms in (\ref{e3}) are inverse to each other, the corresponding two terms in (\ref{e4}) are inverse to each other. (\ref{e3}) means that we can simplify the right hand-side to $e_g$ by cancelling adjacent inverse terms. So, we can also simplify the right hand-side of (\ref{e4}) to $\phi_i(g)$ by the same procedure as in (\ref{e3}). 
	
	It remains to show that there exists $\epsilon_i \to 0$ such that $\phi_i(g)$ is $\epsilon_i$-close to the identity on $\bar{B}_{1/250}(\tilde{p}_i)$, that is, $d(\phi_i(g)z,z) \le \epsilon_i$ for all $z \in \bar{B}_{1/250}(\tilde{p}_i)$. If this is true, then $\phi_i(g)$ converges to the identity map on $\bar{B}_{1/250}(\tilde{p})$ and thus on $\bar{B}_{1/10}(\tilde{p})$ as well by lemma \ref{lemma_nofix}; together with $\phi_i(g)$ converging to $g$, this proves $g$ is the identity. 
	
	On $\Gamma_i$ the association law always holds; we consider $$h_{ij}:= \phi_i(g_{j1})\phi_i(g_{j2})\phi_i(g_{j1}g_{j2})^{-1}$$ first. Since $\phi_i$ are equivariant GH approximations, for any large $i$ and any $z \in B_{1/250}(\tilde{p}_i)$, $\phi_i(g_{j1})\phi_i(g_{j2})(z)$ is close to $\phi_i(g_{j1}g_{j2})(z)$. So $h_{ij}(z)$ is close to $z$. Then $h_{ij}$ is close to identity map $1$ on $\bar{B}_{1/250}(\tilde{p}_i)$ and $h_{ij} \in G_i$. Then $h_{ij}$ must be close to identity map $1$ on $\bar{B}_{1/10}(\tilde{p}_i)$ by Corollary \ref{Cor1}. 
	
	We claim that $h^{-1}h_{ij}h$ is close to $1$ on $\bar{B}_{1/10}(\tilde{p}_i)$ for large $i$ and all $h \in G_i$. Choose any $\epsilon >0$. We have shown that $d(h_{ij}x,x) \le \epsilon$ for large $i$ and all $x \in \bar{B}_{1/10}(\tilde{p}_i)$. Since $h \in G_i$, $d(h\tilde{p}_i,\tilde{p}_i) \le 1/100$ always holds. Then $hx' \in \bar{B}_{1/10}(\tilde{p}_i)$ for all $x' \in \bar{B}_{1/250}(\tilde{p}_i)$. Let $x=hx' \in  \bar{B}_{1/10}(\tilde{p}_i)$, we have $d(h_{ij}hx',hx') < \epsilon$ for large $i$ and all $x' \in \bar{B}_{1/250}(\tilde{p}_i)$. In particular, $h^{-1}h_{ij}h$ is close to $1$  on $\bar{B}_{1/250}(\tilde{p}_i)$ when $\epsilon$ is small. Using corollary \ref{Cor1}, $h^{-1}h_{ij}h$ is close to $1$ on $\bar{B}_{1/10}(\tilde{p}_i)$ for all large $i$. 
	
	Since $\phi_i(h_j^l)^{s_j^l} \in G_i$, by induction we get 
	$$(\prod_{l=1}^{k_j} \phi_i(g_j^l)^{s_j^l}) (\phi_i(g_{j1})\phi_i(g_{j2})\phi_i(g_{j1}g_{j2})^{-1})^{s_j} (\prod_{l=k_j}^{1} \phi_i(g_j^l)^{-s_j^l})$$ is close to $1$ on $B_{1/250}(\tilde{p}_i)$. Then $\phi_i(g)$ is close to $1$ on $B_{1/250}(\tilde{p}_i)$. Let $i \to \infty$, we have $\phi_i(g)$ converges to the identity map by lemma \ref{lemma_nofix}. 
	
	Generally, if $[e_g]=[e_h]$ for some $g,h \in G$, then it follows from the same method that $g=h$.
\end{proof}

Now we construct a topology on $\hat{G}$. Let 
$$K=\{g'\in G|d(g'\tilde{p},\tilde{p})<1/200\}$$
and let $\{ U_{\lambda} \}_{\lambda\in \Lambda}$ be the set of all open sets of $G$ that is contained in $K$. We use left translations on $\hat{G}$ and the map $i$ defined in Lemma \ref{lem_inj} to define a basis on $\hat{G}$; more precisely,
$$\{\hat{g}\cdot i(U_{\lambda}) | \hat{g} \in \hat{G},\lambda\in \Lambda \}$$ generates a topology on $\hat{G}$.

\begin{lem}\label{topo}
	For any $g \in G$ and $\lambda \in \Lambda$, $gU_{\lambda}$ is a subset of $G'$. Then $\{g U_{\lambda} \cap G |g\in G,\lambda\in\Lambda\}$ generates the compact-open topology of $G$.
\end{lem}
\begin{proof}
	Recall that the compact-open topology of $G$ comes from $G$ being a subset of $C(\bar{B}_{1/10}(\tilde{p}),\bar{B}_{1/3}(\tilde{p}))$. We first show that this is the same topology comes from $G\subset C(\bar{B}_{1/250}(\tilde{p}),\bar{B}_{1/3}(\tilde{p}))$. Since for locally compact metric spaces, compact-open topology is same as uniform convergence topology, it is enough prove that if $g_i \in G$ uniformly converges to $g \in G$ on $\bar{B}_{1/250}(\tilde{p})$, then also on $\bar{B}_{1/10}(\tilde{p})$. By precompactness, we may assume $g_i$ uniformly converges to some $g'$ on $\bar{B}_{1/10}(\tilde{p})$. It follows from Lemma \ref{lemma_nofix} that $g=g'$. From now on, we will use the compact-open topology from $C(\bar{B}_{1/250}(\tilde{p}),\bar{B}_{1/3}(\tilde{p}))$.
	
	Next, we show that $g U_{\lambda} \cap G$ is open in $G$ for any $g \in G$ and $\lambda \in \Lambda$. Regarding this $U_\lambda$, we may assume there exist a compact set $C \subset \bar{B}_{1/250}(\tilde{p})$ and an open set $V \subset \bar{B}_{1/3}(\tilde{p})$ such that $O=\{g' \in G' | g'C \subset V \}$ and  $U_{\lambda}=K \cap O$. Recall that $d(g'\tilde{p},\tilde{p}) < 1/200$ for all $g' \in U_{\lambda} \subset K$, we have $g'C \subset B_{1/100}(\tilde{p})$. So, we can assume $V \subset B_{1/100}(\tilde{p})$ by replacing $V$ by $V\cap B_{1/100}(\tilde{p})$. 
	For any $g \in G$, $gV$ is well-defined. Then 
	\begin{align*}
		gU_{\lambda} = g(K\cap O) & = \{ gg' | g' \in K \subset G, g'C \subset V \} \\
		&= \{ g''\in G' | d(g''\tilde{p}, g\tilde{p})<1/200, g''C \subset gV \},
	\end{align*}
	which is open in $G'$; here we use the fact that $gg' \in G'$ if $g,g' \in G$.
	
	Then we prove that any open set in $G$ can be generated by $\{g U_{\lambda} \cap G |g\in G,\lambda\in\Lambda\}$. As above, we can assume that this open set has the form $O= \{g' \in G' | g'C \subset V \}$, where $C\subset \bar{B}_{1/250}(\tilde{p})$ is compact and $V\subset \bar{B}_{1/3}(\tilde{p})$ is open. Let 
	\begin{equation*}
		T= \{ g' \in G' | d(g'\tilde{p},\tilde{p}) < 1/95 \},
	\end{equation*}
	which is an open set containing $G$. Choose $O'=O\cap T$, then $O\cap G = O'\cap G$. It suffices to show  that $O' \cap G$ can be generated by $\{g U_{\lambda} \cap G |g\in G,\lambda\in\Lambda\}$. Notice that $gO' \subset G'$ for all $g \in G$.  
	
	For any $g' \in O' \subset T$, $g'C' \subset B_{1/40}(\tilde{p})$; therefore, we may assume $V \subset B_{1/40}(\tilde{p})$ and $g^{-1}V$ is well-defined for all $g^{-1} \in G$.  
	Then by the same argument above,
	\begin{equation*}
		g^{-1}O'= \{ g'' \in G' | d(g''\tilde{p}, g^{-1}\tilde{p}) < 1/95, g''C \subset g^{-1}V \} 
	\end{equation*}
	is open in $G'$. By our choice of $\{U_{\lambda} |\lambda\in\Lambda\}$, $g^{-1}O'\cap K$ is in this family of open sets. It is direct to check
	\begin{equation*}
		O' \cap G=\cup_{g \in G} ( g(g^{-1}O' \cap K) \cap G).
	\end{equation*}
	Hence $O' \cap G$ can be generated by $\{g U_{\lambda} \cap G |g\in G,\lambda\in\Lambda\}$.  
\end{proof}

\begin{lem}\label{lem_embed}
	$i:G \to i(G)$ is a homeomorphism. Moreover, the image $i(G)$ contains a neighborhood of $1\in \hat{G}$.
\end{lem}
\begin{proof}
	We have proved injection in Lemma \ref{lem_inj}. Also, $i:G \to i(G)$ is an open map by the construction of topology on $\hat{G}$. We show the continuity of $i$. For any $\hat{g}i(U_{\lambda})$, where $\hat{g}\in\hat{g}$ and $\lambda\in\Lambda$, we shall show that $i^{-1}(\hat{g}i(U_{\lambda}))$ is open in $G$. Assume $i^{-1}(\hat{g}i(U_{\lambda}))$ is non-empty. Let $g\in i^{-1}(\hat{g}i(U_{\lambda}))$.  This means that there exists $g_{\lambda} \in U_{\lambda}$ such that $[e_g]= \hat{g} [e_{g_{\lambda}}] \in \hat{g}i(U_{\lambda}) \cap i(G)$. We need to find an open neighborhood of $g$ contained in $i^{-1}(\hat{g}i(U_{\lambda}))$. Since $U_{\lambda}$ is an open subset of $K$, $U:=gg_{\lambda}^{-1}U_{\lambda} \cap G$ is an open neighborhood of $g$ in $G$ by the same proof of Lemma \ref{topo}.  In $\hat{G}$, we have $$i(U) \subset [e_{g}][e_{g_{\lambda}^{-1}}]i(U_{\lambda}) = \hat{g}[e_{g_{\lambda}}] [e_{g_{\lambda}^{-1}}]i(U_{\lambda}) =\hat{g} i(U_{\lambda}).$$ Thus $U \subset i^{-1}(\hat{g}i(U_{\lambda}))$. 
\end{proof}

Now we show that $\hat{G}$ is a Lie group. By Lemma \ref{lem_embed}, it suffices to show $G$ contains no small groups, i.e., a neighborhood of the identity contained in $G$ has no non-trivial group.

\begin{lem}
	$G$ contains no small groups.
\end{lem}

\begin{proof}
	Suppose that $G$ has a sequence of subgroups $H_i$ such that $D_{H_i,1/10}(\tilde{p}) \to 0$. Similar to the proof of Lemma \ref{lemma_nofix}, we shall find $D_{H_{i(\epsilon)}, r(\epsilon)}(x(\epsilon))= r(\epsilon)/20$ to derive a contradiction.

	By Lemma \ref{lemma_nofix}, there exists a regular point $w_i \in B_{1/100}(\tilde{p})$ such that $h(w_i) \neq w_i$ for some $h \in H_i$. By passing to a subsequence, we may assume $w_i \to w$ and $s = (1/10 - d(w, \tilde{p}))/3 >0$. Consider $U= B_{s}(w)$. For any $x,y \in U$, all geodesics from $x$ to $y$ are contained in $B_{1/10 - s}(\tilde{p})$. 
	
	Let $z \in (\mathcal{R}_k)_{\epsilon, 2\eta} $ be a regular point in $U$, assume $\eta$ is small enough such that $B_{2\eta}(z) \subset U$. Then for large $i$, we have $D_{H_i, \eta}(z) \le \eta/20$. Since $w_i$ is a regular point, there exists $\theta$ such that $w_i \in (\mathcal{R}_k)_{\epsilon, 2\theta}$. We may choose $\theta < d(h(w_i), w_i)$ for one $h \in H_i$ with $h(w_i) \neq w_i$. Then we have $D_{H_i, \theta}(w_i) \ge \theta/20$. Let $\lambda <\min( \theta, \eta)$. By the same proof of Lemma \ref{lemma_nofix}, we can find $z$ such that $D_{H_i,\lambda}(z(\epsilon)) = \lambda/20$, which will lead to a contradiction.
\end{proof}

\section{Construction of a $\hat{G}$-space}
We define 
\begin{equation*}
	\hat{G} \times_G \bar{B}_{1/250}(\tilde{p})= \hat{G} \times \bar{B}_{1/250}(\tilde{p}) / \sim,
\end{equation*} 
where the equivalence relation is given by $(\hat{g}i(g),\tilde{x}) \sim (\hat{g},g\tilde{x})$ for all $g \in G, \hat{g} \in \hat{G}, \tilde{x} \in \bar{B}_{1/250}(\tilde{p})$ with $g\tilde{x} \in \bar{B}_{1/250}(\tilde{p})$. We endow $\hat{G} \times_G \bar{B}_{1/250}(\tilde{p})$ with the quotient topology. We will always write $[\hat{g}, \tilde{x}]$ as elements in $\hat{G} \times_G \bar{B}_{1/250}(\tilde{p})$ and $(\hat{g}, \tilde{x})$ as elements in $\hat{G} \times \bar{B}_{1/250}(\tilde{p})$. By $\hat{g}'[\hat{g},\tilde{x}]=[\hat{g}'\hat{g},\tilde{x}]$, $\hat{G}$ acts on $\hat{G} \times_G \bar{B}_{1/250}(\tilde{p})$.

\begin{lem}\label{homeo}
	$\hat{G}$ acts as homeomorphisms on $\hat{G} \times_G \bar{B}_{1/250}(\tilde{p})$.
\end{lem} 
\begin{proof}
	We fix any element $\hat{g}'\in \hat{G}$. The map
	$$\hat{g}':\hat{G} \times_G \bar{B}_{1/250}(\tilde{p})\to \hat{G} \times_G \bar{B}_{1/250}(\tilde{p}),\quad [\hat{g},\tilde{x}]\mapsto \hat{g}'[\hat{g},\tilde{x}]$$
	is bijective. It suffices to show that this is an open map. Let $U$ be an open subset of $\hat{G} \times_G \bar{B}_{1/250}(\tilde{p})$. Let $\pi: \hat{G} \times \bar{B}_{1/250}(\tilde{p}) \to \hat{G} \times_G \bar{B}_{1/250}(\tilde{p})$ be the quotient map. Note that $\pi$ commutes with $\hat{G}$-action:
	$$\pi(\hat{g}'(\hat{g},\tilde{x}))=[\hat{g}'\hat{g},\tilde{x}]=\hat{g}'[\hat{g},\tilde{x}]=\hat{g}'(\pi(\hat{g},\tilde{x})).$$
	Thus 
	$$\pi^{-1}(\hat{g}'U)=\hat{g}'\pi^{-1}(U).$$
	Since $\hat{g}'\pi^{-1}(U)$ is open in $\hat{G} \times \bar{B}_{1/250}(\tilde{p})$, we conclude that $\hat{g}'U$ is open as well.
\end{proof}

The following lemma translates the study of $\hat{G}$ to $G$ in certain cases.  

\begin{lem} \label{lemma2}
	Suppose $[\hat{g}, \tilde{x}]=[1, \tilde{x}']$ for given $\hat{g} \in \hat{G}, \tilde{x} \in \bar{B}_{1/250}(\tilde{p}), \tilde{x}' \in \bar{B}_{1/250}(\tilde{p})$, then $\hat{g}= [e_g]$ for some $g \in G$ with $g(\tilde{x})=\tilde{x}'$.  
\end{lem}

\begin{proof}
	By $[\hat{g}, \tilde{x}]=[1, \tilde{x}']$ and the definition of $\hat{G} \times_G \bar{B}_{1/250}(\tilde{p})$, there exist $k$ and $g_j \in G$ for $1 \le j \le k$ so that 
	\begin{align*}
		[\hat{g}, \tilde{x}] & = [[e_{g_1}e_{g_2}...e_{g_k}], \tilde{x}] \\
		& = [[e_{g_1}e_{g_2}...e_{g_{k-1}}], g_{k}\tilde{x}] \\
		& = [[e_{g_1}e_{g_2}...e_{g_{k-2}}], g_{k-1}g_{k}\tilde{x}] \\
		& = ... \\
		& = [1, g_1g_2...g_{k-1}g_k \tilde{x}] 
	\end{align*}
and $g_1g_2g_{k-1}g_k \tilde{x}= \tilde{x}'$. Here $g_{k-1}g_{k}\tilde{x}$ means $g_{k-1}(g_{k-1}(\tilde{x}))$. In particular, we have $g_k\tilde{x} \in \bar{B}_{1/250}(\tilde{p})$, $g_{k-1}g_k\tilde{x} \in \bar{B}_{1/250}(\tilde{p})$,..., and $g_1g_2...g_{k-1}g_k \tilde{x} \in \bar{B}_{1/250}(\tilde{p})$. Since $g_{k-1}g_k\tilde{x} \in \bar{B}_{1/250}(\tilde{p})$, we conclude that $g_{k-1}g_k \in G$ by Lemma \ref{lemma4}. Using Lemma \ref{lemma4} repeatedly, we see that $g= g_1g_2...g_{k-1}g_k \in G$ with $g(x)=\tilde{x}'$.
\end{proof}

Now we construct homeomorphisms between opens sets in $\hat{G} \times_G \bar{B}_{1/250}(\tilde{p})$ and the ones in $\bar{B}_{1/250}(\tilde{p})$.

\begin{lem}\label{local homeo}
	For any $\tilde{x}\in {B}_{1/250}(\tilde{p})$, there exists a neighborhood $U$ of  $[1,\tilde{x}]\in \hat{G} \times_G \bar{B}_{1/250}(\tilde{p})$ which is homeomorphic to a neighborhood $U'$ of $\tilde{x}$.
\end{lem}

\begin{proof}
	Let $\epsilon=1/250 - d(\tilde{x},\tilde{p})>0$. Recall that we denote
	$$K=\{g'\in G|d(g'\tilde{p},\tilde{p})<1/200\}.$$
	Choose a small neighbourhood $U_1$ of $1$ in $K$ with $U_1^{-1}=U_1$ and $(U_1)^2 \subset K$; let $U_2\subset B_{1/250}(\tilde{p})$ be a neighborhood $\tilde{x}$.  We can choose $U_1$ and $U_2$ small enough such that $gg'z \in B_{\epsilon/3}(\tilde{x})$ for all $g,g' \in U_1, z \in U_2$. This guarantees that any geodesic between two points in $U_{1}(U_2)=\{gz | g\in U_1, z \in U_2 \}$ is contained in $B_{1/250}(\tilde{p})$. 
	
	Let $\pi: \hat{G} \times \bar{B}_{1/250}(\tilde{p}) \to \hat{G} \times_G \bar{B}_{1/250}(\tilde{p})$ be the quotient map and let $U=\pi (i(U_1) \times U_2)$. $U$ is open since 
	\begin{equation*}
		\pi^{-1}(U) =\bigcup ( i(U_1)i(g)^{-1} \times (g U_2 \cap \bar{B}_{1/250}(\tilde{p})))
	\end{equation*}
	is union of some open sets in $\hat{G} \times \bar{B}_{1/250}(\tilde{p})$, where the union is taken over all $g \in G$ with $g U_2 \cap \bar{B}_{1/250}(\tilde{p}) \neq \emptyset$.
	
	We define a natural map
	$$\phi: \bar{U}=\pi (i(\bar{U}_1) \times \bar{U}_2)\to \bar{U}_1(\bar{U}_2),\quad [i(g),z]\mapsto gz,$$
	where $\bar{U}_i$ means the completion of $U_i$. We shall show that $\phi$ is a homeomorphism. Then we can restrict $\phi$ on $U$ to obtain the desired homeomorphism between $U$ and $U'=U_1(U_2)$.
	
	We first prove $\phi$ is well defined. Assume $[i(g_1),z_1]=[i(g_2),z_2]$ for $g_1,g_2 \in \bar{U}_1$ and $z_1,z_2 \in \bar{U}_2$. Since $U_1(U_2) \subset B_{1/250}(\tilde{p})$, we have $[1,g_1z_1]=[1,g_2z_2]$. By lemma \ref{lemma2}, $\phi([i(g_1),z_1])=g_1z_2=g_2z_2=\phi([i(g_2),z_2])$. 
	
	Next, we show that $\phi$ is bijective. It is clear that $\phi$ is surjective. We prove the injection of $\phi$. Since $\bar{U}_1(\bar{U}_2) \subset B_{1/250}(\tilde{p})$, if $g_1z_1=g_2z_2 \in B_{1/250}(\tilde{p})$ for $g_1,g_2 \in \bar{U}_1$ and $z_1,z_2 \in \bar{U}_2$, we have
	\begin{equation*}
		[i(g_1),z_1]=[1,g_2z_2]=[1,g_1z_1]=[i(g_2),z_2].
	\end{equation*}
	Hence $\phi$ is injective.
	
	$\phi^{-1}$ is continuous: Assume $g_i z_i \to w\in \bar{U}_1(\bar{U}_2)$, where $g_i\in \bar{U}_1$ and $z_i \in \bar{U}_2$. We can write $w=gz$ for some $g \in \bar{U}_1$ and $z \in \bar{U}_2$. By passing to a subsequence, we may assume $g_i \to g' \in \bar{U}_1$ and $z_i \to z' \in \bar{U}_2$. Then $g'z'=w=gz$. By the injection of $\phi$, we see that $[i(g),z]=[i(g'),z']$. Hence 
	\begin{equation*}
		\phi^{-1}(g_iz_i) = [i(g_i),z_i] \to [i(g'),z']=[i(g),z]= \phi^{-1}(w). 
	\end{equation*}
	This shows the continuity of $\phi^{-1}$ on $\bar{U}_{1}(\bar{U}_2)$. 
	
	$\phi$ is continuous: Assume $[i(g_i),z_i] \to [i(g),z] \in \pi(\bar{U}_1 \times \bar{U}_2)$, where $g_i,g \in \bar{U}_1$ and $z_i,z \in \bar{U}_2$. By passing to a subsequence, we can assume $g_i \to g' \in \bar{U}_1$ and $z_i \to z' \in \bar{U}_2$. Then $[i(g_i),z_i] \to [i(g'),z'] = [i(g),z]$. So 
	\begin{equation*}
		\phi([i(g_i),z_i])=g_iz_i \to g'z'  = \phi([i(g'),z'])= \phi([i(g),z]).
	\end{equation*}
	This proves the continuity of $\phi$. 
	
	In conclusion, $\phi$ is the desired homeomorphism.
\end{proof}

We need the following to apply Palais's theorem.

\begin{lem} \label{lemma3}
	For all $[1,\tilde{x}]$ with $d(\tilde{x},\tilde{p})<1/250$, let $U$ be a neighborhood of $[1, \tilde{x}]$ as described in Lemma \ref{local homeo}. Then the set $\{\hat{g} \in \hat{G} | \hat{g}U \cap U \neq \emptyset \}$ has compact closure in $\hat{G}$.
\end{lem}

\begin{proof}
	We write $U$ as $\pi(i(U_1)\times U_2)$ as in the proof of Lemma \ref{local homeo}. We claim that $i:G\to \hat{G}$ provides a homeomorphism between $\{\hat{g} \in \hat{G} | \hat{g}U \cap U \neq \emptyset \}$ and $\{g \in G | gU_2(U_1) \cap U_2(U_1) \neq \emptyset \}$ (see Lemma \ref{lem_embed}). The latter one has compact closure in $G$ because $G$ is compact; consequently, the set $\{ \hat{g} \in \hat{G}| \hat{g}U \cap U \neq \emptyset \}$ has compact closure in $\hat{G}$ as well.
	
	We verify the claim. By Lemma \ref{lem_embed}, it suffices to check that $i$ maps one set exactly to the other. If $\hat{g}U \cap U \neq \emptyset$, then we can find $g_1,g_2 \in U_1$ and $\tilde{x}_1,\tilde{x}_2 \in U_2$ such that $[\hat{g}i(g_1),\tilde{x}_1]=[i(g_2),\tilde{x}_2]$. So $[\hat{g},g_1\tilde{x}_1]=[1,g_2\tilde{x}_2]$. By Lemma \ref{lemma2}, $\hat{g} = [e_g]=i(g)$ for $g \in G$ with $g g_1 \tilde{x}_1 = g_2 \tilde{x}_2$, that is $g$ satisfies $g U_1(U_2) \cap U_1(U_2) \neq \emptyset$. Conversely, if $g U_1(U_2) \cap U_1 (U_2) \neq \emptyset$, then we have $g_1,g_2 \in U_1$ and $\tilde{x}_1,\tilde{x}_2 \in U_2$ such that $gg_1\tilde{x}_1=g_2\tilde{x}_2$. Thus $i(g)[i(g_1),\tilde{x}]=[i(g_2),\tilde{x}]$, which shows $i(g)U \cap U \neq \emptyset$. 
\end{proof}

\begin{cor}\label{cor_slice}
	There is a $\hat{G}_{[1,\tilde{x}]}$-slice $\hat{S}$ at $[1,\tilde{x}]$.
\end{cor}

\begin{proof}
   This follows directly from Lemma \ref{lemma3} and Theorem \ref{topo_slice}.
\end{proof}

\begin{rem}\label{rem_1}
	We can assume that the slice $\hat{S}$ at $[1,\tilde{x}]$ in Corollary is contained in $U$, the neighborhood of $[1,\tilde{x}]$ described in Lemma \ref{local homeo}. To see this, first note that by Lemma \ref{lemma2} that $\hat{G}_{[1,\tilde{x}]}= i(G_{\tilde{x}})$, so $[1, B_r(\tilde{x})]$ is $\hat{G}_{[1,\tilde{x}]}$-invariant for any $r$. Then $\hat{S} \cap [1, B_r(\tilde{x})]$ is still a $\hat{G}_{[1,\tilde{x}]}$-slice. By choosing $r$ sufficiently small, we may assume $\hat{S}$ is contained in $U$. 
\end{rem}

\begin{rem}\label{rem_2}
	By the proofs of Lemmas \ref{lemma2} and \ref{lemma3}, the homeomorphism $\phi$ is equivariant, that is, compatible with the group actions: for any $\hat{g} \in \hat{G}$ and $z \in U$, $\hat{g}z \in U$ iff $\hat{g}=i(g)$ for some $g \in G$ and $g(\phi(z)) \in U'$; moreover $g(\phi(z))=\phi(i(g)z)$. As a result, $U/\hat{G}$ is homeomorphic to $U'/G$. 
\end{rem}

Now we finish the proof of Theorem \ref{pseudo slice}.

\begin{proof}[Proof of Theorem \ref{pseudo slice}]
	We define $\hat{X}=\hat{G} \times_G \bar{B}_{1/250}(\tilde{p})/\hat{G}$ with quotient topology. Let $\tilde{x} \in B_{1/800}(\tilde{p})$. By choosing a smaller neighborhood, we may assume that $U'=\phi(U)$ in Lemma \ref{local homeo} is contained in $B_{1/800}(\tilde{p})$ where $\phi$ is the homeomorphism constructed in Lemma \ref{local homeo}. We will show that $S=\phi(\hat{S})$ is a $G_{\tilde{x}}$-slice at $\tilde{x}$. Since $\phi([1,\tilde{x}])=\tilde{x}$ and $\hat{S}$ is $\hat{G}_{[1,\tilde{x}]}$-invariant, we see that $\tilde{x} \in S$ and $S$ is $G_{\tilde{x}}$-invariant. It remains to show that $S/G_{\tilde{x}}$ is homeomorphic to a neighborhood of $x$. By Remark \ref{rem_2}, $S/G_{\tilde{x}}$ is homeomorphic to $\hat{S}/\hat{G}_{[1,\tilde{x}]}$. The latter one is homeomorphic to a neighborhood $V$ of $[[1,\tilde{x}]] \in  \hat{X}$ by the definition of a slice. Note that $U/\hat{G} \hookrightarrow \hat{X}$ is homeomorphic onto its image which contains $V$. $U/\hat{G}$ is homeomorphic to $U'/G$. $U'/G$ is homeomorphic to a neighborhood of $x$ in $B_{1/800}(p) \subset \bar{B}_{1/250}(\tilde{p})/G$ (also see Remark \ref{observation}). Combining all these homeomorphisms, we get that $S/G_{\tilde{x}}$ is homeomorphic to a neighborhood of $x$.   
\end{proof}

With Theorem \ref{pseudo slice}, we finish the proof of Theorem \ref{semi}, which follows the same idea as in Section 2.3.

\begin{proof}[Proof of Theorem \ref{semi}]
	Let $B_1(\tilde{p_i})$ be the $1$-ball in $\widetilde{B_4(p_i)}$ and let 
	$$G_i=\{g\in \Gamma_i| d(g\tilde{p_i},\tilde{p_i})\le 1/100\}.$$
	As seen in Section 3, we can pass to a subsequence and obtain convergence
	$$(\bar{B}_1(\tilde{p}_i),\tilde{x}_i,G_i) \overset{GH}\longrightarrow (\bar{B}_1(\tilde{p}),\tilde{x},G), \ \
(\bar{B}_1(p_i),x_i) \overset{GH}\longrightarrow (\bar{B}_1(p),x),$$
	where $x\in B_{1/800}(p)$ and $\tilde{x}\in B_{1/800}(\tilde{p})$ with $\pi(\tilde{x})={x}$. By Theorem \ref{PW} and the assumption that $B_1(\tilde{p}_i)$ is non-collapsing, for any small $\epsilon>0$, there is $\delta>0$ such that $\rho(t,\tilde{x})\le\epsilon$ for all $t\le \delta'$; in other words, any loop in $B_{\delta'}(x)$ is contractible in $B_\epsilon(x)$.
	
	By Theorem \ref{pseudo slice}, there is a slice $S\subset B_{1/800}(\tilde{p})$ at $\tilde{x}$ such that $S/G_{\tilde{x}}$ is homeomorphic to a neighborhood of $x$. Using this homeomorphism and the same argument as in Section 2.3, we can find $\delta>0$ such that $B_\delta(x)\subset B_{1/800}(x)$ and $\pi^{-1}(B_\delta(x))\cap S \subset B_{\delta'}(\tilde{x})$. For any loop $\gamma$ in $B_\delta(x)$, we can make use of the homeomorphism to lift $\gamma$ to a loop $\tilde{\gamma}$ in $\pi^{-1}(B_\delta(x))\cap S \subset B_\delta'(\tilde{x})$. For this $\tilde{\gamma}$, there is a homotopy $H$, whose image is in $B_\epsilon(\tilde{x})$, that contracts $\tilde{\gamma}$. Then $\pi\circ H$ is a homotopy that contracts $\gamma$ with $\mathrm{im}H\subset B_\epsilon(x)$. This shows that $\rho(t,x)\le \epsilon$ for all $t\le \delta$. Thus $\lim_{t\to 0}\rho(t,x)=0$.
\end{proof}

\section{Proof of Theorem \ref{desc_cover}}

We define $G^r$ by the $G$-action on $Y$ as below: 
\begin{center}
	$G^r=\langle \{ g\in G \ |\ d(q,gq)\le r$ for some $q\in Y\}\rangle$.
\end{center}
Note that by definition, $G^r$ is normal in $G$ for any $r$. We first prove the following gap lemma for $G^r$. One may compare this with \cite{KapovitchWilking2011} (see Lemma 2.4 and the 3rd paragraph in the proof of Theorem 9.1).

\begin{lem} \label{gap}
	There exists a number $\delta>0$ such that $G^r=G^\delta$ for all $r\in(0,\delta]$. Moreover, for this $\delta$, $G^\delta$ is the group generated by $G_0$ and all isotropy subgroups of $G$.
\end{lem}

We first prove the single point version. We define
\begin{center}
	$G(r,q)=\langle\{g\in G\ |\ d(q,gq)\le r\}\rangle$.
\end{center}

\begin{lem} \label{gap_pt}
	Let $q\in Y$. There exists a number $\delta=\delta(q)>0$ such that $G(r,q)=G(\delta,q)$ for all $r\in(0,\delta]$. Moreover, for this $\delta>0$, $G(\delta,q)$ is the group generated by $G_0$ and $G_q$, the isotropy subgroup at $q$.
\end{lem}

\begin{proof}
	By Theorem \ref{isom_lie}, $G$ is a Lie group, thus each connected component of the orbit $G\cdot q$ is both open and closed in $G\cdot q$. Let $C_q$ be the component containing $q$. Pick $\delta>0$ sufficiently small such that $B_{\delta}(q)\cap G\cdot q \subset C_q$. For any $0<r\le\delta$, by our choice of $\delta$, we get $G(r,q)\cdot q \subset C_q$. It is also clear that $G(r,q)\cdot q$ is both open and closed in $G\cdot q$. Therefore, we conclude that $G(r,q)\cdot q=C_q$. 
	
	Now we check that $G(r,q)=G(\delta,q)$. It is clear that $G(r,q)\subset G(\delta,q)$. On the other hand, given any $g\in G(\delta,q)$, then $gq \in C_q$. Since $G(r,q)\cdot q=C_q$, we get $gq=hq$ for some $h\in G(r,q)$. Notice that $h^{-1}g$ fixes $q$, thus $h^{-1}g\in G(r,q)$. Consequently, $g=h\cdot(h^{-1}g)\in G(r,q)$.
	
	For the second part of the Lemma, we denote $H_q=\langle G_0, G_q \rangle$.
	Since $G_0$ is generated by elements in any small neighborhood of the identity element, we have $H_q\subset G(\delta,q)$. Conversely, for any element $g\in G(\delta,q)$, we have shown that $gq\in G(r,q)\cdot q=C_q=G_0\cdot q$. Then by the same argument is the former paragraph, $g$ can be written as $h\cdot (h^{-1}g)$ with $h\in G_0$ and $h^{-1}g\in G_q$.
\end{proof}

Now we prove Lemma \ref{gap}:
\begin{proof}[Proof of Lemma \ref{gap}]
	We first prove that $G^r$ stabilizes as $r\to 0$. We argue by contradiction. Suppose that there is $r_i\to 0$ such that $G^{r_{i+1}}$ is a proper subgroup of $G^{r_i}$ for each $i$. We pick $g_i\in G^{r_i}-G^{r_{i+1}}$ with $d(g_iq_i,q_i)\le r_i$ for some $q_i \in Y$. Let $H$ be the group generated by $G_0$ and all isotropy subgroups of $G$. It is clear that $H$ is a subgroup of $G^r$ for any $r>0$. Hence each $g_i$ is outside $H$. Let $C_{i}$ be the connected component of $G\cdot q_i$ that contains $q_i$. Since $g_i$ cannot be generated by $G_0$ and $G_{q_i}$, 
	$g_iq_i$ must be outside $C_i$. Let $\overline{X}$ be closure of a lift of $X$ in $Y$ and let $h_i\in G$ such that $q'_i=h_iq_i$ is a point in $\overline{X}\cap (G\cdot q_i)$. Then $g'_i=h_ig_ih^{-1}_i$ satisfies 
	$$d(g'_iq'_i,q'_i)=d(h_ig_iq_i,h_iq_i)=d(g_iq_i,q_i)\le r_i.$$
	Recall that $G^r$ is normal for each $r$, thus $g'_i\in  G^{r_i}-G^{r_{i+1}}$. After passing to a subsequence, $q_i$ converges to $q\in\overline{X}$ and $g'_i$ converges to $g\in G$ with $gq=q$. Let $\delta(q)>0$ as in Lemma \ref{gap_pt}. By Lemma \ref{gap_pt}, $g_i'\in G(\delta(q),q)=G(r,q)$ for all $r\le \delta(q)$. It follows that $g'_i\in H$ for $i$ sufficiently large. A contradiction.
	
	For the second half of the statement, the proof is similar to the one of Lemma \ref{gap_pt}. Let $H$ be the group generated by $G_0$ and all isotropy subgroups of $G$. It is clear that $H$ is a subgroup of $G^\delta$. Conversely, let $g$ be a generator of $G^\delta$. There is $q\in Y$ with $gq\subset C_q= G_0\cdot q$ Therefore, $g$ can be generated by elements in $G_0$ and $G_q$.
\end{proof}

With Lemma \ref{gap} and \cite[Theorem 3.10]{FukayaYamaguchi1992}, we are ready to prove Theorem \ref{desc_cover}.

\begin{proof}[Proof of Theorem \ref{desc_cover}]
	We first recall that it was shown in \cite{EnnisWei2006} that there is $\epsilon_0>0$ such that for all $\epsilon\le\epsilon_0$, $Y/H^\epsilon$ covers the universal cover of $X$, where $H^\epsilon$ comes from 
	$$(\widetilde{M}_i,\tilde{p},\Gamma^\epsilon_i)\overset{GH}\longrightarrow (Y,\tilde{p},H^\epsilon).$$
	Hence $Y/H^\epsilon$ is the universal cover of $X$.
	
	Let $\epsilon_0$ be the constant above and $\delta$ be the constant in Lemma \ref{gap}. Put $\eta=\min\{\epsilon_0,\delta\}/2$. Since $G^\eta$ contains $G_0$, $G/G^\eta$ is discrete. Because $X=Y/G$ is compact, we can apply \cite[Theorem 3.10]{FukayaYamaguchi1992} to find a sequence of subgroups $K_i$ converging to $G^\eta$. According to \cite{FukayaYamaguchi1992}, $K_i$ can be constructed as follows. Let $(f_i,\phi_i,\psi_i)$ represent an $\epsilon_i$ pointed equivariant Gromov-Hausdorff approximation $(\widetilde{M}_i,\tilde{p}_i,\Gamma_i)\to(Y,\tilde{p},G)$ for some $\epsilon_i\to 0$. Then we have a map $\phi_i:S_i(1/\epsilon_i)\to S_\infty(1/\epsilon_i)$, where
	\begin{center}
		$S_i(r)=\{g\in\Gamma_i\ |\ d(\tilde{p}_i,g\tilde{p}_i)\le r\}$.
	\end{center}
	Let $R$ be a sufficiently large number compared with $3D$, then $K_i$ is generated by $\{g\in S_i(R)\ |\ \phi_i(g)\in G^\eta\}$. 
	
	We claim that $K_i=\Gamma_i^\eta$ for all $i$ large. Let $g_i$ be a generator of $\Gamma_i^\eta$, that is, $g_i$ has displacement less than $\eta$ at certain point. Consequently, $\phi_i(g_i)\in G^{\eta+\epsilon_i}=G^\eta$. The last equality follows from our choice of $\eta$ and Lemma \ref{gap}. This shows that $\Gamma_i^\eta\subset K_i$. Conversely, suppose that there is a sequence $g_i\in S_i(R)$ with $\phi_i(g_i)\in G^\eta$. 
	Since $G^\eta = G^{\eta/4}$, we can write $\phi_i(g_i)=\prod_{j=1}^{l_i} g_{i,j}$, where $g_{i,j}$ are generators of $G^{\eta/4}$. Here we can assume that $l_i$ is bounded by a constant number depending on $R$ but not $i$ because $G^{\eta}$ is generated by $G_0$ and isotropy subgroups. Let $h_{i,j}\in\Gamma_i$ that is the element closest to $g_{i,j}\in G^{\eta/4}$, then $h_{i,j} \in \Gamma_i^{\eta/2}$ for $i$ large. Note that $\prod_{j=1}^{l_i} h_{i,j}$ and $\prod_{j=1}^{l_i} g_{i,j}$ share the same limit in $G$ as $i\to\infty$. We conclude that $g_i \in \Gamma_i^{\eta}$ for all $i$ large. This proves the claim that $\Gamma_i^\eta = K_i$.
	
	So far, we have shown that $(\widetilde{M}_i,\tilde{p}_i,\Gamma_i^\eta)\overset{GH}\longrightarrow(Y,\tilde{p},G^\eta)$. Together with the second part of Lemma \ref{gap} and the result in \cite{EnnisWei2006} as explained above, we finish the proof.
\end{proof}

If we further assume that $M_i$ is non-collapsing, then $G$ is discrete.
Thus
\begin{cor}
	Under the assumptions of Theorem \ref{desc_cover}, if in addition there is $v>0$ such that
	$\mathrm{vol}(M_i)\ge v,$
	then $H$ is the group generated by all isotropy subgroups of $G$.
\end{cor}	
	
\bibliographystyle{plain} 
\bibliography{temp}
\end{document}